\documentclass[11pt,a4paper]{article}
\usepackage{amssymb} 
\usepackage{amsmath} 
\usepackage{amsthm} 
\usepackage{hyperref}
\usepackage{a4wide}

\usepackage{subcaption,tikz}
\usetikzlibrary{arrows, fit,  calc, decorations.markings, decorations.pathmorphing, shapes}

\newtheorem{thr}{Theorem}
\newtheorem{remark}[thr]{Remark}

\newtheorem{conjecture}[thr]{Conjecture}
\newtheorem{defi}[thr]{Definition}
\newtheorem{q}[thr]{Question}

\newtheorem{prop}[thr]{Proposition}
\newtheorem{examp}{Example}

\newtheorem{claim}{Claim}

\newcommand{\floor}[1]{\left\lfloor#1\right\rfloor}

\newcommand{\h}{\mathcal{H}}
\newcommand{\sizeof}[1]{\lvert#1\rvert}

\DeclareMathOperator{\ex}{ex}
\DeclareMathOperator{\rex}{rex}

\title{Regular Tur\'an numbers and some Gan--Loh--Sudakov-type problems}

\author{
Stijn Cambie
\thanks{Department of Mathematics, Radboud University Nijmegen, Netherlands. 
	Supported by a Vidi grant (639.032.614) of the Netherlands Organisation for Scientific Research (NWO). Email addresses: \protect\href{mailto:stijn.cambie@hotmail.com}{\protect\nolinkurl{stijn.cambie@hotmail.com}}, \protect\href{mailto:r.deverclos@math.ru.nl}{\protect\nolinkurl{r.deverclos@math.ru.nl}}, \protect\href{mailto:ross.kang@gmail.com}{\protect\nolinkurl{ross.kang@gmail.com}}.}
\and
R\'emi de Joannis de Verclos\footnotemark[1]
\and
Ross J. Kang\footnotemark[1]
}

\begin{document}
	
	\definecolor{xdxdff}{rgb}{0.49019607843137253,0.49019607843137253,1.}
	\definecolor{ududff}{rgb}{0.30196078431372547,0.30196078431372547,1.}
	
	\tikzstyle{every node}=[circle, draw, fill=black!50,
	inner sep=0pt, minimum width=4pt]

\maketitle
\parindent=4pt

\begin{abstract}
	Motivated by a Gan--Loh--Sudakov-type problem, we introduce the regular Tur\'an numbers, a natural variation on the classical Tur\'an numbers for which the host graph is required to be regular.	
	Among other results, we prove a striking supersaturation version of Mantel's theorem in the case of a regular host graph of odd order.
	We also characterise the graphs for which the regular Tur\'an numbers behave classically or otherwise.
\end{abstract}

\section{Introduction}\label{sec:intro}

Mantel's theorem~\cite{Man07}, Tur\'an's theorem~\cite{Tur41} and the Erd\H{o}s--Stone theorem~\cite{ES46} are foundational in extremal graph theory. The extremal graphs seem close to being regular. For example, the construction of $C_4$-free graphs with many edges using Sidon sets (see e.g.~\cite{DTF18}) is regular for even order and has difference $1$ between minimum and maximum degree for odd order. Nevertheless, the restriction to regular graphs forces that the maximum number of edges, or, equivalently, maximum degree, when avoiding certain $3$-chromatic graphs depends heavily on the parity of the order.
To make this phenomenon more concrete, we use the following terminology.

\begin{defi}
The {\em regular Tur\'an number} of a graph $H$ is
$$\ex_r(n, H) = \max\{k: \lvert V (G)\rvert = n, G\mbox{ is }k\mbox{-regular and does not contain }H\mbox{ as a subgraph}\}.$$
For a family of graphs $\h$, $\ex_r(n, \h)$ is defined similarly, so $G$ must not contain any $ H \in \h $.
\end{defi}

The following result, focusing on odd cycles, is most illustrative.

\begin{thr}\label{thr:ex_r_OddCycles}
	For fixed $\ell\ge 3$ and $H=C_{2\ell-1}$, it holds for sufficiently large $n$ that
	$$\ex_r(n, H)=\begin{cases}
	 \frac n2 & \mbox{ if n is even}\\
	2 \left \lfloor \frac n{2\ell+1} \right \rfloor & \mbox{ if n is odd.}		\end{cases}$$
\end{thr}

This is a generalisation of a regular version of Mantel's theorem and we prove it in Section~\ref{sec:exr_oddcycles}.

With an extremal result in hand, one may naturally pursue supersaturation results, as in e.g.~\cite{Erd62,LS83, Raz08, Rei16}.
The classical such result for Mantel's theorem states that the minimum number of triangles in a graph is linear in the number of additional edges. More precisely, if a graph has at least $\left \lfloor \frac{n^2}4 \right\rfloor +\ell$ edges for some $0\le \ell < \frac n2$, then it must contain at least $g_3(t_2(n)+\ell)=\ell \left \lfloor \frac n2 \right \rfloor$ triangles.
If a graph has at least $ \gamma n^2$ edges for $\gamma > \frac 14$, then the graph contains at least $\Omega_{\gamma}(n^3)$ triangles.
In particular if the average degree is $\frac n2 + 1$, there are at least $\Omega(n^2)$ triangles, while there are at least $\Omega(n^3)$ triangles when the average degree is at least $(0.5+ \gamma)n$ for some constant $\gamma > 0$.

On the other hand, supersaturation differs for regular Mantel's theorem, as the minimum number of triangles in a $k$-regular graph with odd order $n$ is $\Theta(n^2)$ for every $\ex_r(n,C_3)<k< \ex_r(n,C_3)+\frac n{10}$. The following is shown in Section~\ref{sec:satRegMantel}.
\begin{thr}[Supersaturated regular Mantel's theorem]\label{thr:quadraticsaturationK3}
	Let $G$ be a $k$-regular graph on $n$ vertices.
	If $n$ is odd and $k>2\floor{\frac{n}{5}}$,
	then $G$ contains at least $\frac{1}{300}n^2$ triangles.
\end{thr}

Returning to regular Tur\'an numbers, we show in Section~\ref{sec:reg_genThrs} that for those $H$ for which $\chi(H)\not=3$ these numbers behave as in the classical Erd\H{o}s--Stone theorem, i.e.~$\lim_{n\to\infty}\frac{1}{n}\ex_r(n,H) = 1-\frac{1}{\chi(H)-1}.$ There remains the question of which $H$ with $\chi(H)=3$ also behave classically, and so not like in Theorem~\ref{thr:ex_r_OddCycles}. In Theorem~\ref{thr:RegularErdosStone3} below, we consider two constructions ---which roughly speaking are adaptations from complete bipartite graphs so that they have an odd number of vertices and are regular--- and prove that they completely describe such $H$.

We describe our point of inspiration for the regular Tur\'an numbers in Section~\ref{sec:givenNM}, that is, some variations on a conjecture of Gan, Loh and Sudakov~\cite{GLS} (now confirmed, cf.~\cite{C19}). In particular, we will consider a variation for which {\em both} order and size are prescribed, after a work by Kirsch and Radcliffe~\cite{KR}.
Although it does not behave as nicely as the original problem, the regular case (which one might consider as the most interesting of these variations) can be almost completely resolved with regular Tur\'an numbers.
A few other possible generalisations of the problem initiated in~\cite{GLS} are posed and briefly discussed in Section~\ref{sec:otherGLStypeQ}.

\section{Regular Tur\'an numbers of odd cycles}\label{sec:exr_oddcycles}

The relationship between the minimum degree of a graph and the existence of cycles or paths of certain lengths has been extensively studied. We begin by listing a few key results which are useful for determining the regular Tur\'an numbers of odd cycles.

\begin{thr}[Andr\'{a}sfai, Erd\H{o}s and S\'os~\cite{AES74}]\label{thr:AES}
Let $\ell \ge 1$ and $G$ be a non-bipartite graph with minimum degree $\delta > 2 \frac n{2\ell+1}$, then $G$ contains an odd cycle $C_m$ with $m \le 2\ell -1.$
In particular, if $G$ is a non-bipartite graph with minimum degree $\delta > 2 \frac n5$, then $G$ contains a triangle.
\end{thr}

\begin{thr}[Voss and Zuluaga~\cite{VZ77}]\label{thr:VZ77}
	Every $2$-connected non-bipartite graph with minimum degree $\delta$ has an odd cycle of length at least $\min\{2\delta -1, n\}.$
\end{thr}

\begin{thr}[H\"{a}ggkvist~\cite{Hag82}]\label{thr:Haggkvist2}
	Let $G$ be a graph with minimum degree $\delta > \frac{2n}{2\ell+1}$ and $n>\binom{\ell+1}{2}(2\ell+1)(3\ell-1).$
	Then either $G$ contains a $C_{2\ell-1}$ or it does not contain any odd cycle $C_m$ for some $m > \frac{\ell}2.$
\end{thr}

\begin{thr}[Liu and Ma~\cite{LM18}]\label{thr:LM18}
	Let $G$ be a $2$-connected bipartite graph, $u,v$ two distinct vertices of $G$ and $d$ the minimum degree of the vertices in $G \backslash \{u,v\}$. Then there is a path between $u$ and $v$ of length at least $2(d-1).$
\end{thr}

\subsection{Regular Tur\'an number of the triangle for all orders}

Here we investigate the same problem as in~\cite{Bau84}, but from the standpoint of the order $n$ instead of the regularity $k$.

\begin{thr}[Regular Mantel's theorem]\label{thr:regularMantel}
	Let $G$ be a $k$-regular, triangle-free graph on $n$ vertices. When $n$ is even, we have $k \le \frac n2.$ When $n$ is odd, we have $k \le 2 \lfloor \frac n5 \rfloor.$ Moreover, these bounds are sharp. Put in another way,
	$$\ex_r(n, K_3) =\begin{cases}
		\frac n2 & \mbox{ if n is even}\\
		2 \lfloor \frac n5 \rfloor & \mbox{ if n is odd}.
	\end{cases}$$
\end{thr}

\begin{proof}
	When $n$ is even, the result follows from the classical Mantel's theorem and complete bipartite graphs achieving equality.
	When $n$ is odd, since $G$ is regular, it cannot be bipartite.

	By Theorem~\ref{thr:AES}, we know $k \le \frac 25 n.$
	Due to the handshaking lemma, we know $k$ has to be even and hence $k \le 2 \lfloor \frac n5 \rfloor$ follows.
	
	Now we show sharpness of the result.	
	Let $n =5x+y,$ with $0 \le y \le 4$ and $y<x.$
	Let $S_1$, $S_2$, $S_3$, $S_4$ and~$S_5$ be stable sets of respective sizes $x+y$, $x$, $x-y$, $x$ and $x+y$.
	Add all edges between vertices of $S_i$ and $S_{i+1}$ for $1 \le i \le 4$ and let $G[S_1,S_5]$ be a $x$-regular bipartite graph (which can be obtained by removing $y$ disjoint complete matchings of a complete bipartite graph $K_{x+y,x+y}$).
	Then the resulting graph $G$ is a $k$-regular, triangle-free graph with $k=2x=2 \lfloor \frac n5 \rfloor$.
	Note that it is a classic blow-up of a $5$-cycle when $5$ divides $n$.
	
	In the remaining cases we have $n \le 19$ (as we consider only odd $n$)  and $n \not=15.$
	Let the vertices be $1$ up to $n$ and connect $i$ and $j$ if $ i-j \equiv \pm (2h+1) \pmod n $ for some $0 \le h \le \lfloor \frac n5 \rfloor-1.$ This results in a $k$-regular, triangle-free graph with $k=2 \lfloor \frac n5 \rfloor.$
	For this note that $3( 2 \lfloor \frac n5 \rfloor -1) < n$ and no three numbers can have pairwise odd differences.
\end{proof}

\subsection{Asymptotic regular Tur\'an numbers of odd cycles}

We next show that the magnitude of $\ex_r(n,\h)$ for different parity can differ by any factor. This is analogous to the main result in~\cite{Zha94}, but focusing on order instead of regularity.

\begin{thr}\label{thr:ex_r_odd_girth}
	For the family $\h=\{C_3,C_5\ldots C_{2\ell -1}\}$, we have
		$$\ex_r(n, \h) =\begin{cases}
		\frac n2 & \mbox{ if n is even}\\
		2 \left \lfloor \frac n{2\ell+1} \right \rfloor -o(1) & \mbox{ if n is  odd.}		\end{cases}$$
\end{thr}

\begin{proof}
	When $n$ is odd, since $G$ is regular, it cannot be bipartite.
	By Theorem~\ref{thr:AES}, we know $k \le \frac 2{2 \ell +1} n.$
	Due to the handshaking lemma, we know $k$ has to be even and hence $k \le 2 \lfloor \frac n{2 \ell +1} \rfloor$ follows.

	Hence the main part is to show sharpness for large $n,$ which
	will be obtained by taking a construction close to the blow-up of a $C_{2\ell+1}.$
	Let $M=2\ell+1$. Let $n=(2\ell+1)x+y$ with $x>y$ and $0 \le y \le 2\ell.$
	Take $M=2\ell+1$ stable sets $S_1, S_2, \ldots , S_M$.
	
	If $\ell$ is odd, equivalently $M \equiv 3 \pmod 4$, we take them such that
	 such that
	$$\lvert S_i \rvert = \begin{cases}
	x & \mbox{ if } i\mbox{ is odd}\\
	x+y & \mbox{ if } i\equiv 2 \pmod 4\\
	x-y & \mbox{ if } i \equiv 0 \pmod 4.		\end{cases}$$
	If $\ell$ is even, equivalently $M \equiv 1 \pmod 4$, we take their sizes to be
	$$\lvert S_i \rvert = \begin{cases}
	x & \mbox{ if } i\mbox{ is even}\\
	x+y & \mbox{ if } i\equiv 1 \pmod 4\\
	x-y & \mbox{ if } i\equiv 3 \pmod 4.		\end{cases}$$
	For $1 \le i \le M$, connect every vertex in $S_i$ with every vertex in $S_{i+1}$, where the indices are taken modulo $M$ and remove $y$ disjoint perfect matchings between $S_1$ and $S_M.$
	Now it is clear that the resulting graph is $2x$-regular, has odd girth $M$ and order $Mx+y=n$.
\end{proof}

As we stated earlier in Theorem~\ref{thr:ex_r_OddCycles}, it suffices for $n$ large enough to exclude only the cycle $C_{2\ell-1}.$

To show this, we require the following.

\begin{remark}
	For every unicyclic graph $H$ having girth $2\ell-1$, if a graph $G$ with minimum degree larger than $\lvert H \rvert$ contains a $C_{2\ell-1}$, then it contains $H$ as well, so $\ex_r(n, H)=\ex_r(n, C_{2\ell-1})$ for large $n$.
\end{remark}

\begin{thr}\label{thr:containing_odd_cycle}
	Let $G$ be a graph of order $n>\binom{\ell+1}{2}(2\ell+1)(3\ell-1) $ with minimum degree  $\delta > \frac{2n}{2\ell+1},$ such that deleting any collection of at most $\ell^3$ edges of $G$ does not result in a bipartite graph. Then $G$ contains a $C_{2\ell-1}.$
\end{thr}

\begin{proof}
	By Theorem~\ref{thr:Haggkvist2}, a counterexample $G$ would not contain any odd cycle $C_m$ for some $m \ge  \frac{\ell}2.$
	First, iteratively, select cutvertices of the resulting graph and delete them. 
	Note that one can have deleted at most $\ell-1$ cutvertices at the end due to the minimum degree condition.
	Once having selected $\ell$ cutvertices there are at least $\ell+1$ components, each contains a vertex of degree at least  $\frac{2n}{2\ell+1}-\ell$, so this is a lower bound on its order.
	But now we would get that there the union of these components is at least $\frac{2(\ell+1)n}{2\ell+1}-\ell^2>n$, contradiction.
	
	For every component, look to the original part of the graph containing that component and the cutvertices adjacent to it. Every such part has at least $\frac{2n}{2\ell+1}+1$ vertices, as it contains vertices (all vertices which were not a cutvertex) of degree at least $\frac{2n}{2\ell+1}.$
	In such a part, iteratively delete every (cut)vertex with minimum degree less than $\ell.$
	Note that no non-cutvertex can be deleted since $\delta >2 \ell.$
	The resulting part is $2$-connected and has minimum degree at least $\ell.$
	So if such a resulting part is non-bipartite, we can find a cycle of length larger than $\ell$ by Theorem~\ref{thr:VZ77}, which gives the desired contradiction.

	Note that in total, we have deleted at most $\ell(\ell-1)$ edges in a single part and there are at most $\ell$ such parts. There are also no more than $\binom{\ell}2$ edges between cutvertices.
	Let $G'$ be the graph obtained by deleting all the edges mentioned before.
	Since there are deleted at most $\ell^2(\ell-1)+\binom{\ell}2 \le \ell^3$ edges, $G'$ is not a bipartite graph by the given assumption. 
	But since every part of $G'$ is bipartite, there is an odd cycle which passes through multiple parts and hence cutvertices.
	Take the one containing the fewest number of cutvertices. This one will enter and leave every part exactly once.
	If not, the intersection of the odd cycle and a part contains at least $2$ disjoint paths.
	Take a shortest path connecting two of these shortest paths. In the odd cycle, this connecting path divides the cycle in two, one of them being of odd length. But that path contains a smaller number of the cutvertices, from which the conclusion follows.

	Take one such part $H$ which has at least one edge in common with the smallest odd cycle and such that exactly two of its cutvertices $u,v$ are on the odd cycle. 
	By Theorem~\ref{thr:LM18} we can find a path in $H$ between $u$ and $v$ of length at least $2\left(\delta - \ell -1 \right) > \ell.$
	Since the length of every path between $u$ and $v$ will have the same parity, replacing the part of the odd cycle between $u$ and $v$ with this path gives the desired contradiction.
\end{proof}

\begin{proof}[Proof of Theorem~\ref{thr:ex_r_OddCycles}]
	Note that if $n$ is odd, one needs to delete at least $\frac k2$ edges from a $k$-regular graph on $n$ vertices to obtain a bipartite graph.
	So if $n>\binom{\ell+1}{2}(2\ell+1)(3\ell-1)$ and $ k> \frac{2n}{2\ell+1}$, the graph contains a $C_{2\ell-1}$ by Theorem~\ref{thr:containing_odd_cycle}.
	Again by the handshaking lemma, we know $k$ is even and thus at most $2 \left \lfloor \frac n{2\ell+1} \right \rfloor.$ Sharpness for large $n$ is due to the construction in the proof of Theorem~\ref{thr:ex_r_odd_girth}.
\end{proof}

\section{Supersaturation for regular Mantel's theorem}\label{sec:satRegMantel}

\begin{proof}[Proof of Theorem~\ref{thr:quadraticsaturationK3}]
  First note that $k\geq\frac{2n}{5}+\frac{2}{5}$
  because $k$ is even
  and $k/2 \geq \floor{\frac{n}{5}}+1\geq \frac{n}{5}+\frac{1}{5}$.

  We start with the following observation.
  \begin{claim}\label{claim:triangle 1/5n intersection}
    Every triangle of $G$ intersects at least $\frac{1}{5}n-2$
    other triangles on an edge.
  \end{claim}
  \begin{proof}
    Let $t$ be the number of triangles that intersect the triangle formed by the vertices
    $u_1$, $u_2$ and $u_3$ on an edge.
    Note that the number of triangles (including $u_1u_2u_3$) that contains the edge $u_iu_j$
    is equal to $|N(u_i)\cap N(u_j)|$,
    so $t = |N(u_1)\cap N(u_2)| + |N(u_2)\cap N(u_3)| + |N(u_3)\cap N(u_1)| - 3$.
    By the inclusion--exclusion principle,
    \[
    n \geq |N(u_1) \cup N(u_2) \cup N(u_3)|
    \geq 3k - (t + 3) \geq \frac{6}{5}n +\frac{6}{5} - t - 3.
    \]
    It follows that $t > \frac{n}{5} - 2$.
  \end{proof}
  Let~$S$ be the set of vertices of~$G$ that are contained
  in at least $\frac{1}{10}n$ triangles.
  It follows from the definition that the total number of triangles
  is at least $\frac{1}{3}\cdot\frac{n}{10}\cdot|S|$,
  so the result holds if $|S|\geq\frac{n}{10}$.
  In the following, we assume that $|S|<\frac{n}{10}$.
  \begin{claim}\label{claim:2 from S}
    Every triangle of~$G$ contains at least two vertices of~$S$.
  \end{claim}
  \begin{proof}
    If otherwise, there is a triangle $T$ with two vertices $u_1$ and $u_2$
    that both intersect fewer than $\frac{1}{10}n$ triangles each, and then
    as every triangle intersecting $T$ on an edge contains (at least)
    one of $u_1$ and $u_2$, it follows that $T$ intersects
    fewer than $2\cdot(\frac{1}{10}n-1)=\frac{1}{5}n-2$
    other triangles on an edge, which
    contradicts Claim~\ref{claim:triangle 1/5n intersection}.
  \end{proof}
  \begin{claim}\label{claim:C5}
    $G\setminus S$ contains no $C_5$ as a subgraph.
  \end{claim}
  \begin{proof}
    Assume for a contradiction that
    $C$ is a $5$-cycle of $G\setminus S$.
    By Claim~\ref{claim:2 from S},
    we know that no triangle of $G$ contains an edge
    of $G\setminus S$.
    It follows that $N_G(u)\cap N_G(v)=\varnothing$
    whenever $uv$ is an edge of $G$.
    As a consequence, a vertex $v$ of $G$ is adjacent to
    at most $\alpha(C_5)=2$ vertices of $C$.

    It follows by double-counting that
    \[
    |N(C)| \geq \frac{5k}{2} > n,
    \]
    which yields a contradiction as $|N(C)|\leq n$.
  \end{proof}
  \begin{claim}\label{claim:bipartite}
    $G\setminus S$ is bipartite.
  \end{claim}
  \begin{proof}
    Assume $G\setminus S$ is not bipartite and let us prove
    the lower bound on~$|S|$.
    Let $u_1\dots u_{2m+1}$ be an odd cycle of $G\setminus S$ with minimal
    length $2m+1$.
    We know from the triangle-freeness of $G\setminus S$
    and Claim~\ref{claim:C5} that $2m+1\geq 7$.
    Let us estimate the size of $N(\{u_1,u_2,u_{m+2}\})$ in $G$.

    First note that $u_1$ and $u_2$ have no common neighbour in $G$
    because of Claim~\ref{claim:2 from S}.
    Moreover, for $i\in\{1,2\}$ the vertices $u_{m+2}$ and $u_i$ cannot
    have a common neighbour in $G\setminus S$ because it would
    yield cycles of lengths $m+2$ or $m+3$ in $G\setminus S$.
    As one of these lengths is odd, this would contradict the minimality
    of $m$.
    As a consequence, $(N(u_1)\cup N(u_2))\cap N(u_{m+2})$
    is a subset of $S$.

    By the inclusion--exclusion principle,
    $3k - |S|\leq |N(\{u_1,u_2,u_{m+2}\})| \leq n $,
    so
    \[
    |S| \geq (3k - n) > \frac{1}{5}n,
    \]
    which contradicts our hypothesis.
  \end{proof}
  Let $A\cup B$ be a bipartition of $G\setminus S$.
  \begin{claim}
    There is a partition $S=S_1\cup S_2$ such that there is no edge between
    $S_2$ and $B$, and no edge between $S_1$ and $A$.
  \end{claim}
  \begin{proof}
    If otherwise, then there are vertices~$a\in A$, $b\in B$
    and $s\in S$ such that $as$ and~$bs$ are edges.
    As every triangle of $G$ contains at least two vertices of~$S$,
    the vertex $a$ has no neighbor in $N(s)\cap B$.
    Similarly, the vertex $b$ has no neighbor in $N(S)\cap A$.
    It follows that $N(s)\cap N(a)$, $N(s)\cap N(b)$
    and $N(a)\cap N(b)$ are included in $S$.
    As a consequence,
    \[
    3k \leq |N(a)|+|N(b)|+|N(s)| \leq n + 2|S|.
    \]
    It follows that $|S|\geq \frac{1}{2}(3k - n) > \frac{n}{10}$,
    which contradicts our hypothesis on $|S|$.
  \end{proof}
  We may assume by symmetry that $|A\cup S_1|\leq |B\cup S_2|$.
  As $|A\cup S_1| + |B\cup S_2|=n$ and $n$ is odd,
  it follows that $|A\cup S_1| \leq (n-1)/2$
  and $|B\cup S_2| \geq (n+1)/2$.
  \begin{claim}
    The number $e_2$ of edges in the induced subgraph
    $G[S_2]$ is at least $k/2$.
  \end{claim}
  \begin{proof}
    Let $e$ denote the number of edges from $A\cup S_1$ to
    $B\cup S_2$.
    It holds that $e \leq k\cdot|A\cup S_1|\leq k(n-1)/2$
    and $2e_2 + e = k\cdot |B\cup S_2|\geq k(n+1)/2$.
    As a consequence,
    \[
    2e_2 \geq k(n+1)/2 - k(n-1)/2 = k,
    \]
    which proves the claim.
  \end{proof}
  We are now ready to conclude the proof.

  First note that every edge $uv$ of $G[S_2]$ is contained in at least
  $\frac{n}{5}$ triangles of $G$.
  Indeed, we know that $N(u)$ and $N(v)$ are subsets of $A\cup S$,
  so
  \[
  |N(u) \cap N(v) | \geq |N(u)|+|N(v)|-|A|-|S|
  \geq 2k - \frac{n-1}{2} - \frac{n}{10}
  > \frac{n}{5}. 
  \]
  
  As there are at least $k/2$ such edges and a triangle
  contains at most three of them, we conclude that the number of
  triangles in $G$ is at least
  \[
  \frac{1}{3}\cdot\frac{k}{2}\cdot\frac{n}{5}
  > \frac{n^2}{75}.\qedhere
  \]
\end{proof}

We can show slightly more.

\begin{thr}\label{thr:SaturationRegMantel}
	When $n$ is odd and $k$ is an even number with $2 \lfloor \frac n5 \rfloor < k \le 2\lfloor \frac n4 \rfloor,$ every $k$-regular graph on $n$ vertices has $\Omega(n^2)$ triangles. Moreover, this is sharp up to the multiplicative constant.
\end{thr}

\begin{proof}
	The lower bound is proven in Theorem~\ref{thr:quadraticsaturationK3}.
	So now we prove sharpness of the result. Let $n=2x+1$ and $2 \lfloor \frac n5 \rfloor<k=x-y\le  2\lfloor \frac n4 \rfloor.$ We construct a $k$-regular graphs with $O(n^2)$ triangles.
	For this take a $K_{x,x}$, delete $y$ disjoint complete matchings and delete another disjoint matching of size $\frac k2$ (i.e.~on $k$ vertices). Now connect the endvertices of that last matching with an additional vertex $v$.
	Every triangle in the resulting graph contains the additional vertex $v$, from which the conclusion follows.
	For $k=2 \lfloor \frac n5 \rfloor+2$, this gives a construction with approximately $\frac{n^2}{50}$ triangles.
\end{proof}

\section{A regular version for all nonbipartite graphs}\label{sec:reg_genThrs}

\begin{thr}\label{thr:construction(r-1)partite}
	Let $r \ge 4$. For every $n \in \mathbb N$, there exists a $k$-regular $(r-1)$-partite graph with $k=\left( 1- \frac1{r-1} + o(1) \right) n.$
\end{thr}

\begin{proof}
	Write $n=(r-1)x+y$ with $x$ even and $0 \le y \le 2r-3.$
	As the statement is an asymptotic one, we only have to deal with $n$ large and so we can assume $(r-2)x>y.$
	Construct a complete $(r-2)$-partite graph $K_{x,x,\ldots, x}$ and remove a $y$-factor of it. 
	Now connect all edges between a stable set of size $x+y$ and all vertices of this graph.
	The resulting graph is a $(r-2)x$-regular graph on $n$ vertices.
\end{proof}

As a corollary to Theorem~\ref{thr:construction(r-1)partite}, the conclusions for the regular versions of Tur\'an's theorem and the Erd\H{o}s--Stone theorem are unchanged from their classical forms, if the chromatic number of the forbidden graph $H$ satisfies $\chi(H)\not= 3.$

\begin{thr}[Regular Erd\H{o}s--Stone theorem for $\chi(H)
\not= 3$]\label{thr:RegularErdosStonenot3}
	Let $H$ be a graph with $\chi(H)\not=3.$ Then	
	$$\ex_r(n, H)=\left(1-\dfrac 1 {\chi(H)-1}+o(1) \right) n.$$
\end{thr}

We already saw in Theorem~\ref{thr:ex_r_OddCycles} that there are graphs $H$ with $\chi(H)=3$ for which the regular Erd\H{o}s--Stone theorem differs from the classical statement.
Next we characterise all such graphs $H$ (with $\chi(H)=3$).

We denote with $K_{2x,y}^=$ a complete bipartite graph $K_{2x,y}$ with a perfect matching in the part of size $2x$.

\begin{thr}[Regular Erd\H{o}s--Stone theorem for $\chi(H)= 3$]\label{thr:RegularErdosStone3}
	Let $H$ be a graph with $\chi(H)=3$.
	\begin{enumerate}
	\item\label{itm:i}
	Suppose one of the following holds:
	\begin{itemize}
	\item for every vertex $v$ of~$H$, the graph~$H\backslash v$
          is not bipartite; or
        \item $H$ is not a subgraph of $K_{2\lvert H \rvert, \lvert H \rvert}^=$.
	\end{itemize}	 
	Then $\ex_r(n,H)= \frac n2 +o(n).$
	\item\label{itm:ii}
	 If neither of the above hold and $n$ is odd, then $\ex_r(n,H) \le 2\floor{\frac{n}{5}}.$
	\end{enumerate}
\end{thr}

\begin{proof}
We begin with the proof of~\ref{itm:i}.
  The upper bound is a consequence of the Erd\H{o}s--Stone theorem. 
  If~$n$ is even, the lower bound is given by the complete bipartite graph~$K_{\frac n2, \frac n2}$,
  so it is enough to give $k$-regular constructions of $H$-free graphs for odd $n$ and
  $k=\frac n2 +o(n)$.

  We distinguish two cases depending on which condition holds.
\begin{itemize} 
\item In the first case, namely if $H\setminus v$ is not bipartite for every $v\in V(H)$,
  the $k$-regular construction in the proof of Theorem~\ref{thr:SaturationRegMantel} for $k=2  \lfloor \frac n4 \rfloor$ does the job. Indeed, one can remove one vertex from the resulting graph $G$ such that it becomes bipartite, so all its subgraphs also have this property.
\item In the second case,
  let $n=2x+1$ and take $K_{x+1,x}$ with $\lfloor \frac{x+1}2 \rfloor$ disjoint edges added at the stable set of size $x+1$.
  If $x$ is odd, this is a the $(x+1)-$regular graph $K_{x+1,x}^=.$
  If $x$ is even, remove a maximum matching between the vertices of degree $x+1$ to get a $x$-regular graph.
  In both cases, the obtained graph is a subgraph of
  $K_{2a,a}^=$ for some~$a$ and therefore does not does not contain $H$.

\end{itemize}		
\medskip
We proceed to the proof of~\ref{itm:ii}.
  Fix an odd number $n$ and let $k >2\floor{\frac{n}{5}}$.
  Set $t=\lvert H \rvert$.
  Let~$G$ be a $k$-regular graph without $H$ as an induced subgraph.
  It follows from this last hypothesis that every neighbourhood~$N(u)$ in~$G$ contains
  no~$H\setminus v$ as a subgraph,
  and therefore no $K_{t,t}$ because $H\setminus v$ is bipartite.
  By the K\"{o}vari--S\'{o}s--Tur\'{a}n Theorem~\cite{KST54},
  it follows that~$G[N(u)]$ contains at most
  $\frac12(t-1)^{\frac{1}{t}}k^{2-\frac{1}{t}}+O(k)$ edges,
  which is smaller than $n^{2-\epsilon}$ if $\epsilon = \frac{1}{t}$ and~$n$ large enough.
  Equivalently, every vertex of~$G$ is contained in fewer than $n^{2-\epsilon}$ triangles.
  
  We say that an edge of $G$ is \emph{thick} if it is contained in at least $\frac{n}{15}$
  triangle.
  Let us show that $G$ contains a set $A\subseteq E(G)$ of
  $\Omega(n^{\epsilon})$ disjoint thick edges.
  
  We first proceed as in Claim~\ref{claim:triangle 1/5n intersection}
  in the proof of Theorem~\ref{thr:quadraticsaturationK3}
  to show that every triangle of~$G$ contains a thick edge.
  Indeed, consider a triangle $u_1u_2u_3$ in $G$. By symmetry
  we may assume
  that $\sizeof{N(u_1)\cap N(u_2)} \geq \sizeof{N(u_2)\cap N(u_3)} \geq \sizeof{N(u_1)\cap N(u_3)}$. 
  By the inclusion--exclusion principle,
	\[
	n \geq \sizeof{N(u_1)\cup N(u_2)\cup N(u_3)} > 3k - 3\sizeof{N(u_1)\cap N(u_2)}
	\]
	and thus it follows that $\sizeof{N(u_1)\cap N(u_2)} > \frac{n}{15}$, so $u_1u_2$ is thick.
	
	Let~$T$ be a maximal set of vertex-disjoint triangles.
	Since~$T$ is maximal, every triangle of~$G$ intersects a vertex of a triangle of~$T$.
	As each vertex is contained in at most $n^{2-\epsilon}$ triangles,
	it then follows from the hypothesis that $G$ has at most
        $3\sizeof{T}\cdot n^{2-\epsilon}$ triangles.
        Theorem~\ref{thr:quadraticsaturationK3} applied to $G$ then yields
	\[
	3|T| \cdot n^{2-\epsilon} \geq \frac{1}{300}n^2,
	\]
	so $\sizeof{T} \geq \frac{n^\epsilon}{900}$.
	
	It then suffices to construct~$A$ by choosing a thick edge in each triangle of~$T$.
	
	To conclude the proof, consider the bipartite auxiliary graph $F$
	on $V(F) = A \cup V$ such that
	for every $e \in A$ and $v \in V$, the pair
	$\{e, v\}$ is an edge of $F$ if and only if $e\cup\{v\}$ is a triangle.
	Note that in particular $v$ is not an endpoint of $e$.
	As the edges of $A$ are thick, every $e\in A$ has degree at least $\frac{n}{15}$ in $F$,
        so $\sizeof{E(F)} \geq \frac{n\sizeof{A}}{15}$.
        By the asymmetric version of the K\H{o}vari--S\'{o}s--Tur\'{a}n
        Theorem~\cite{KST54,Znam63},
        this implies the existence of a $K_{t,t}$ as a subgraph of~$F$ provided that
        \[
        (t-1)^{\frac{1}{t}}|A|n^{1-\frac{1}{t}}+(t-1)n < n\frac{\sizeof{A}}{15},
        \]
        which is true whenever $n$ --and therefore $A$-- is large enough.
        Since a copy of $K_{t,t}$ in $F$ gives a copy of $K_{2t, t}^=$ in~$G$,
        the graph~$G$ contains a copy of $H$,
        which yields a contradiction and concludes the proof.
\end{proof}

In Theorem~\ref{thr:RegularErdosStone3}\ref{itm:ii}, there is possibly still room for improvement on the value of $\ex_r(n,H)$.
Based on Theorem~\ref{thr:ex_r_OddCycles}, it is natural to wonder if the value depends on the odd girth.
More precisely, the following question would be worth investigating.

\begin{q}\label{q:dich_girth_ge5}
Let $H$ be a graph with $\chi(H)=3$ such that there exists a vertex $v$ for which $H \backslash v$ is bipartite and such that $H$ is a subgraph of $K_{2\lvert H \rvert, \lvert H \rvert}^=.$
	Let the odd girth of the graph $H$ be $g$.
	It is true that $\ex_r(n,H) = 2 \left \lfloor \frac n{g+2} \right \rfloor+o(n) ?$
\end{q}

	\section{Maximizing clique count given order and size}\label{sec:givenNM}

Given a graph $G$, define $k_t(G)$ to be the number of cliques $K_t$ in $G$.
Gan, Loh and Sudakov~\cite{GLS} proposed the problem of maximizing $k_t(G)$ in $G$ given the order and the maximum degree of $G$.
Motivated by the Gan--Loh--Sudakov problem, Kirsch and Radcliffe~\cite{KR} proposed the problem of maximizing $k_t(G)$ in $G$ given the size and the maximum degree of $G$.
They also wondered about the problem of maximizing $k_t(G)$ in $G$ given the order $n$ {\em and} size $m$, as well as the maximum degree of $G$.
(Up to the maximum degree condition, this question appeared for example also in~\cite{GNPV}.) As will become apparent, the most interesting case here is that of regular graphs. This case is closely related to the Kahn--Zhao theorem~\cite{Kahn01,Zhao10}, which is a natural predecessor to the Gan--Loh--Sudakov problem.
When the order $n$ is not much larger than the degree $r$ of the regular graphs, by focusing on the complementary graph $\overline G$, some cases are related to a conjecture of Kahn~\cite{Kahn01}. This is the case when $2(n-r)\mid n$. For this note that $k_t(G)=i_t( \overline G)$, where $i_t$ is the number of independent sets of order $t$, as every clique in $G$ is an independent set in $\overline G$ and vice versa.

Chase~\cite{C19} solved the main problem in~\cite{GLS}.
Chakraborti and Chen~\cite{CC20} recently solved (in a stronger form) the main conjecture in~\cite{KR}, and due to this the extremal graph for our problem with $n=a(r+1)+b$ (here $b \le r$) and $m \le a \binom{r+1}2 + \binom{b}2$ is the union of $K_{r+1}$s and a colex graph. For the remaining cases, i.e.~in the critical regime (where one cannot have $a$ $K_{r+1}$'s), it is natural to pose the following analogous conjecture.

	\begin{conjecture}\label{conj:maxgivenNandM}
		Let $n=a(r+1)+b$ and $\frac{nr}2 \ge m>a \binom{r+1}2 + \binom{b}2.$
		Any graph maximizing $k_t$ for a fixed $t$ or $k=\sum_{t \ge 2} k_t$ among all graphs of order $n$, size $m$ and maximum degree at most $r$ can be represented as $(a-1)K_{r+1}+H$.

\end{conjecture}
There are some obstructions to a tidier conjecture.
Examples~\ref{examp:k3vsk} and~\ref{examp:multipleH} show that there might be several different kinds of extremal graph $H$, and for distinct $t$ the extremal graphs might not correspond. This is in stark contrast to the cases of prescribed size or order alone.

\begin{examp}\label{examp:k3vsk}
	The graph $G$ in Figure~\ref{fig:k3max} satisfies $k_3(G)=16, k_4(G)=4, k_5(G)=0$ and $k(G)=20.$ It is the unique graph maximizing $k_3(G)$ among all graphs with $(n,m,r)=(8,18,5).$
	On the other hand, the graph $G$ in Figure~\ref{fig:k5max} satisfies $k_3(G)=15, k_4(G)=6, k_5(G)=1$ and $k(G)=22.$ It is the unique graph maximizing $k(G)$ among all graphs with $(n,m,r)=(8,18,5)$ and maximizes $k_4$ and $k_5$ as well.
	For $k_4$ and $k_5$ there are respectively $2$ and $3$ extremal graphs.	
	
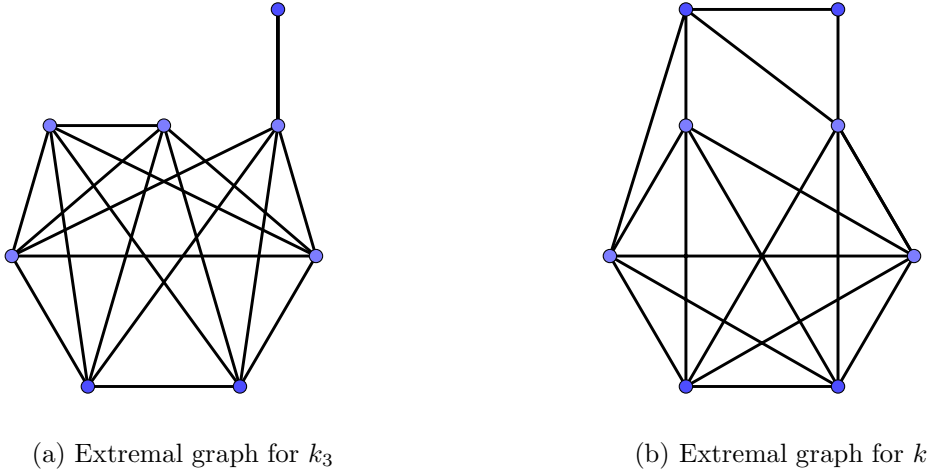
\begin{figure}[h]
	\begin{minipage}[b]{.5\linewidth}
		\begin{center}
			
			\begin{tikzpicture}[line cap=round,line join=round,>=triangle 45,x=1.0cm,y=1.0cm]
			\clip(-0.5,-0.5) rectangle (5,6);
			\draw [line width=1.1pt] (1,0)--(3,0)--(4,1.73)--(0,1.73)--(1,0);
			\draw [line width=1.1pt] (1,0)-- (0.5 ,3.46) -- (3,0);
			\draw [line width=1.1pt] (4,1.73)--(0.5 ,3.46)--(0,1.73);
			\draw [line width=1.1pt] (1,0)-- (2 ,3.46) -- (3,0);
			\draw [line width=1.1pt] (4,1.73)--( 2 ,3.46)--(0,1.73);
			\draw [line width=1.1pt] (1,0)-- ( 3.5,3.46) -- (3,0);
			\draw [line width=1.1pt] (4,1.73)--(3.5 ,3.46)--(0,1.73);			
			\draw [line width=1.1pt]  (2,3.46)--(0.5,3.46);
			\draw [line width=1.1pt]  (3.5,3.46)--(3.5,5)--(3.5,3.46);
			\begin{scriptsize}
			\draw [fill=ududff] (1,0) circle (2.5pt);
			\draw [fill=xdxdff] (0, 1.73) circle (2.5pt);
			\draw [fill=ududff] (3.,0) circle (2.5pt);
			\draw [fill=xdxdff] (4,1.73) circle (2.5pt);
			\draw [fill=xdxdff] (0.5,3.46) circle (2.5pt);
			\draw [fill=xdxdff] (2,3.46) circle (2.5pt);
			\draw [fill=xdxdff] (3.5,3.46) circle (2.5pt);
			\draw [fill=ududff] (3.5,5) circle (2.5pt);
			\end{scriptsize}
		\end{tikzpicture}
		\subcaption{Extremal graph for $k_3$} \label{fig:k3max}
	\end{center}
\end{minipage}
\begin{minipage}[b]{.5\linewidth}
\begin{center}	
	
\begin{tikzpicture}[line cap=round,line join=round,>=triangle 45,x=1.0cm,y=1.0cm]
\clip(-0.5,-0.5) rectangle (5,6);
\draw [line width=1.1pt] (1,0)--(3,0)--(4,1.73)--(1,3.46)-- (0,1.73)--(1,0)--(4,1.73)--(0,1.73)--(3,0)--(1,3.46)--(1,0) ;
\draw [line width=1.1pt] (3,3.46)--(3.,5)-- (1,5)--(3,3.46);
\draw [line width=1.1pt] (4,1.73)--(3,3.46)-- (3,0);
\draw [line width=1.1pt] (1,0)-- (3,3.46)--(4,1.73);
\draw [line width=1.1pt] (1,3.46)-- (1,5)--(0,1.73);
\begin{scriptsize}
\draw [fill=ududff] (1,0) circle (2.5pt);
\draw [fill=xdxdff] (0, 1.73) circle (2.5pt);
\draw [fill=ududff] (3.,0) circle (2.5pt);
\draw [fill=xdxdff] (4,1.73) circle (2.5pt);
\draw [fill=xdxdff] (1,3.46) circle (2.5pt);
\draw [fill=xdxdff] (3,3.46) circle (2.5pt);
\draw [fill=ududff] (1,5) circle (2.5pt);
\draw [fill=ududff] (3,5) circle (2.5pt);
\end{scriptsize}
\end{tikzpicture}
\subcaption{Extremal graph for $k$} \label{fig:k5max}
\end{center}
\end{minipage}	
\caption{Graphs with $(n,m,r)=(8,18,5)$}
\label{fig:weirdexamples}
\end{figure}
	
\end{examp}

As the $t=3$ case was the main interest in~\cite{KR}, we can further focus on this case.

The following equality expresses $k_3(G)$ in terms of its order, the degrees and $k_3( \overline{G}).$ It is basically proven in~\cite{Good}.
\begin{claim}\label{claim:good}
	For any graph $G$ of order $n$, we have
	$$k_3(G)+k_3( \overline{G})+\frac 12 \sum_v \deg(v)(n-1-\deg(v)) = \binom{n}{3}.$$
\end{claim}
	
Describing the extremal graphs in general seems to be hard as they are not unique and also $k_3( \overline G )$ and the degree sequences can be different for different extremal graphs, as the next example shows.

\begin{examp}\label{examp:multipleH}
	There are three graphs with the maximum number of triangles, $16$, among all graphs of order $8$, size $17$ and maximum degree at most $5$. 
	The number of triangles in their complement $\overline G$ is equal to $4, 1$ and $0$ respectively, implying also that their degree sequences are different.
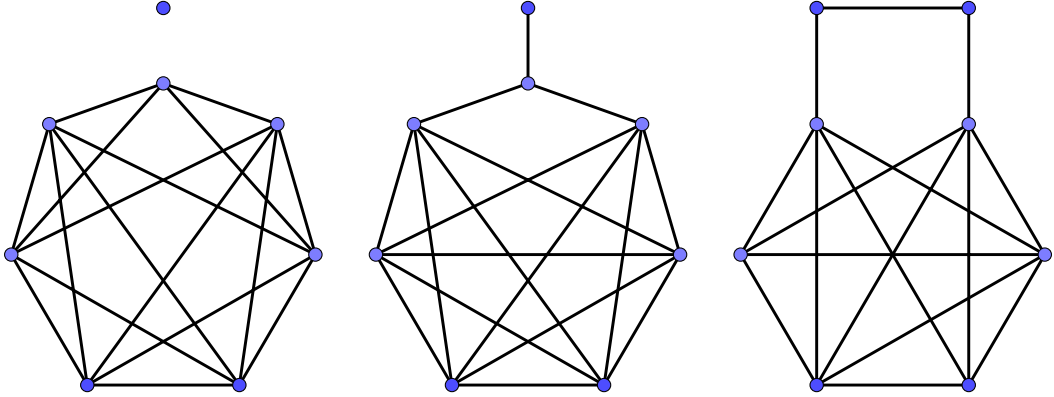
\begin{figure}[h]
	\begin{minipage}[b]{.3\linewidth}
		\begin{center}
			
			\begin{tikzpicture}[line cap=round,line join=round,>=triangle 45,x=1.0cm,y=1.0cm]
			\clip(-0.5,-0.5) rectangle (5,6);
			
			\node (1) at (1,0) {};
			\node (2) at (3.,0) {};
			\node (3) at (4,1.73) {};
			\node (4) at (3.5,3.46){};
			\node (5) at (2,4){};
			\node (6) at (0.5,3.46){};
			\node (7) at (0, 1.73){};
			\draw [line width=1.1pt] (1)--(3,0)--(4,1.73)--(3.5,3.46)--(2,4)--(0.5,3.46)--(0,1.73)--(1,0);
			\draw [line width=1.1pt] (6)--(1)--(3)--(5)--(7)--(2)--(4);
			\draw [line width=1.1pt] (1)--(4)--(7);
			\draw [line width=1.1pt] (3)--(6)--(2);
			\begin{scriptsize}
			\draw [fill=ududff] (1,0) circle (2.5pt);
			\draw [fill=xdxdff] (0, 1.73) circle (2.5pt);
			\draw [fill=ududff] (3.,0) circle (2.5pt);
			\draw [fill=xdxdff] (4,1.73) circle (2.5pt);
			\draw [fill=xdxdff] (0.5,3.46) circle (2.5pt);
			\draw [fill=xdxdff] (2,4) circle (2.5pt);
			\draw [fill=xdxdff] (3.5,3.46) circle (2.5pt);
			\draw [fill=ududff] (2,5) circle (2.5pt);
			\end{scriptsize}
			\end{tikzpicture}
		\end{center}
	\end{minipage}
\begin{minipage}[b]{.3\linewidth}
	\begin{center}
		
		\begin{tikzpicture}[line cap=round,line join=round,>=triangle 45,x=1.0cm,y=1.0cm]
		\clip(-0.5,-0.5) rectangle (5,6);
			\node (1) at (1,0) {};
			\node (2) at (3.,0) {};
			\node (3) at (4,1.73) {};
			\node (4) at (3.5,3.46){};
			\node (5) at (2,4){};
			\node (6) at (0.5,3.46){};
			\node (7) at (0, 1.73){};
			\node (8) at (2,5){};
			\draw [line width=1.1pt] (1)--(3,0)--(4,1.73)--(3.5,3.46)--(2,4)--(0.5,3.46)--(0,1.73)--(1,0);
			\draw [line width=1.1pt] (7)--(2)--(4);
			\draw [line width=1.1pt] (6)--(1)--(3);
			\draw [line width=1.1pt] (1)--(4)--(7)--(3)--(6)--(2);

			\draw [line width=1.1pt] (5)--(8);
			\begin{scriptsize}
			\draw [fill=ududff] (1,0) circle (2.5pt);
			\draw [fill=xdxdff] (0, 1.73) circle (2.5pt);
			\draw [fill=ududff] (3.,0) circle (2.5pt);
			\draw [fill=xdxdff] (4,1.73) circle (2.5pt);
			\draw [fill=xdxdff] (0.5,3.46) circle (2.5pt);
			\draw [fill=xdxdff] (2,4) circle (2.5pt);
			\draw [fill=xdxdff] (3.5,3.46) circle (2.5pt);
			\draw [fill=ududff] (2,5) circle (2.5pt);
		\end{scriptsize}
		\end{tikzpicture}
	\end{center}
\end{minipage}
	\begin{minipage}[b]{.3\linewidth}
		\begin{center}	
			
			\begin{tikzpicture}[line cap=round,line join=round,>=triangle 45,x=1.0cm,y=1.0cm]
			\clip(-0.5,-0.5) rectangle (5,6);
			
			\node (1) at (1,0) {};
			\node (2) at (3.,0) {};
			\node (3) at (4,1.73) {};
			\node (4) at (3,3.46){};
			\node (5) at (1,3.46){};
			\node (6) at (0,1.73){};
			\node (7) at (3, 5){};
			\node (8) at (1,5){};
			
			\draw [line width=1.1pt] (1)--(2)--(3)--(4);
			\draw [line width=1.1pt] (5)--(6)--(1)--(3)--(5)--(1)--(4);
			\draw [line width=1.1pt] (2)--(4)--(6)--(3);
			\draw [line width=1.1pt] (2)--(5)--(8)--(7)--(4);
			\begin{scriptsize}
			\draw [fill=ududff] (1,0) circle (2.5pt);
			\draw [fill=xdxdff] (0, 1.73) circle (2.5pt);
			\draw [fill=ududff] (3.,0) circle (2.5pt);
			\draw [fill=xdxdff] (4,1.73) circle (2.5pt);
			\draw [fill=xdxdff] (1,3.46) circle (2.5pt);
			\draw [fill=xdxdff] (3,3.46) circle (2.5pt);
			\draw [fill=ududff] (1,5) circle (2.5pt);
			\draw [fill=ududff] (3,5) circle (2.5pt);
			\end{scriptsize}
			\end{tikzpicture}
		\end{center}
	\end{minipage}	
	\caption{Graphs with $(n,m,r)=(8,17,5)$ maximizing $k_3$}
	\label{fig:weirdexamples2}
\end{figure}
\end{examp}

We also remark that in the critical regime, increasing $m$ can imply both a decrease or increase in the number of triangles. This is also the case if one increases both $m$ and $n$ by $1$.

\begin{examp}
	When $r=4$, the maximum number of triangles among all graphs of order $n$ and size $m$ in the critical regime are given below in Table~\ref{table:maxk3_nm}.
	
		\begin{table}[h]
			\begin{center}
			\begin{tabular}{ |c | c  c c c c c  | }
				\hline			
				$n \backslash m$ & 11 & 12 & 13&14&15 &16 \\
				\hline
				6 & 7&8&&&& \\
				7 &  & 8&7&7&& \\
				8 &  & & &8&8&8 \\
				\hline  
			\end{tabular}
			\end{center}
			\caption{maximum $k_3(G)$ given $n$ and $m$ when $r=4$}\label{table:maxk3_nm}
		\end{table}
\end{examp}

We also give some positive results, e.g.~we can describe the extremal graphs for $n=r+2.$

\begin{prop}
	When $n=r+2$ and $\binom{r+1}2 \le m \le \frac{(r+2)r}2$, the extremal graph $G$ attaining the maximum number of triangles among all graphs of order $n$, size $m$ and maximum degree at most $r$ is the one for which $ \overline G$ is the union of a matching of size $m-\binom{r+1}2$ and a star of order $(r+1)^2+1-2m$.
\end{prop}

\begin{proof}
	Note that the degrees $0 \le d_i \le r$ satisfy $\sum d_i =2m$ and the function $f(x)=x(n-1-x)$ is a strictly concave function. By the inequality of Karamata~\cite{Kar}, this implies that $\sum f(d_i) \ge (r+1)f(r)+f(2m-r(r+1)).$
	By Claim~\ref{claim:good} this implies that $k_3(G) \le \binom{n}3 - (r+1)f(r)-f(2m-r(r+1))$. Equality occurs if $k_3( \overline G)=0$ and $\overline G$ has $r+1$ vertices of degree $1$ and one vertex of degree $(r+1)^2-2m$, from which the characterization follows.	
\end{proof}

For $n=r+3$ one can get a similar characterization, up to a few exceptions, if $m \ge  \frac{(r+2)r}2$. In that case, the complement $\overline G$ is in general a union of cycles, some of them having one vertex in common.

The case where $m=\frac{nr}2$, i.e.~when the graphs are regular, might be considered as the most interesting case because it is the extreme case which is most far apart from the edge case. Due to Conjecture~\ref{conj:maxgivenNandM} we focus on this case for $r+2 \le n \le 2r+1.$
\begin{thr}
	Let $r+2 \le n \le 2r+1$ and $m=\frac{nr}2$. Every graph $G$ maximizing the number of triangles among the graphs of order $n$ and size $m$ can be formed by taking the complement of a $(n-r-1)$-regular graph on $n$ vertices minimizing the number of triangles.
	In particular, if $n$ is even or $n \le r +1+ 2 \lfloor \frac r3 \rfloor $ is odd and the maximum number of triangles equals $k_3(G)=\binom{n}{3}-\frac n2 r(n-1-r) $. 
\end{thr}

\begin{proof}
	By Claim~\ref{claim:good} we know $k_3(G)=\binom{n}{3}-\frac n2 r(n-1-r) -k_3( \overline{G}).$
	So the maximum is attained if the $(n-r-1)$-regular graph on $n$ vertices $\overline{G}$ minimizes $k_3$.
	So Theorem~\ref{thr:regularMantel} implies the exact result for $n$ being even (take $\overline G$ bipartite) or $n$ being odd and $n \le r+1 + 2 \lfloor \frac r3 \rfloor $. Section~\ref{sec:satRegMantel} implies $k_3(G)=\binom{n}{3}-\frac n2 r(n-1-r)-\Theta(n^2)$ in the remaining case. 
\end{proof}

	
The exact result for the regular case would be known once proven Conjecture~\ref{conj:maxgivenNandM} and the following conjecture.
	
\begin{conjecture}\label{conj:exactSat}
	Let $G$ be a $k$-regular graph on $n$ vertices, with $n=2p+1$ being odd and $2 \lfloor \frac n5 \rfloor < k=p-q \le 2 \lfloor \frac n4 \rfloor$ being even.
	Then $k_3(G) \ge \frac{k}2 \left( \frac{k}2-q-1 \right).$
\end{conjecture}

Equality for this conjecture holds when $G$ is a graph formed by a $K_{p,p}$ and an additional vertex $v$ connected to $\frac{k}2$ vertices of both stable sets of $K_{p,p}$ and deleting edges between the $k$ neighbours of $v$ on the one hand and between the $2p-k$ nonneighbours of $v$ on the other hand, such that the final graph is $k$-regular.

We prove one case of this conjecture.

\begin{prop}
	Let $G$ be a $k$-regular graph on $n=2k+1$ vertices, with $k$ being even.
	Then $k_3(G) \ge \frac{k}2 \left( \frac{k}2-1 \right)$. Equality holds if and only if $G$ is a $K_{k,k}$ minus a matching of size $\frac k2$, where all endvertices of the matching are connected with an additional vertex.
\end{prop}

\begin{proof}
	Assume there is such a graph $G$ with $k_3(G) < \frac{k}2 \left( \frac{k}2-1 \right).$
	Take a vertex $v$ of $G$ for which the number of triangles containing $v$, $k_3(v)$, is minimal. In particular $x=k_3(v)< \frac{3}{2k+1} \frac{k}2 \left( \frac{k}2-1 \right) < \frac{k}2-1.$
	Since $G=(V,E)$ is $k$-regular, $\lvert N(v) \rvert = k$ and  $\lvert N_2(v) \rvert = k$, where $N_2(v)=V \backslash N[v]$. There are $x$ edges in $G[N(v)]$ and $x+\frac k2$ edges in $G[N_2(v)]$.
	Note that the two endvertices of any of these $x+\frac k2$ edges in $G[N_2(v)]$ have at least $2k-(x+\frac k2+1)-k=\frac k2-x-1$ common neighbours in $N(v).$ 
	Similarly, every edge among the $x$ edges in $G[N(v)]$ is contained in at least $2k-(x+1)-(k+1)=k-x-2$ triangles.
	This implies that $G$ contains at least $x(k-x-2)+(x+\frac k2)(\frac k2-x-1)=\frac{k}2 \left( \frac{k}2-1 \right)+x(k-2x-3)$ triangles.
	Since $2x \le k-4$, the result follows.
	We attain equality if $k_3(v)=0$ and $G[N_2(v)]$ is a star and the vertices in $N(v)$ are connected to all vertices in $N_2(v)$ except one in such a way that they have total degree $k$, so the extremal graph is of the desired form.
\end{proof}

\section{Some other Gan--Loh--Sudakov-type problems}\label{sec:otherGLStypeQ}

The general Tur\'an-type study of Alon and Shikhelman~\cite{AS} asks to determine the quantity $\ex(n,T,H)$, the maximum number of copies of $T$ in an $H$-free graph on $n$ vertices.
The Gan--Loh--Sudakov problem can be formulated thus as the special case of determining $\ex(n,K_{t},K_{1,r+1})$.
If $n$ is a multiple of $r+1$, it is trivial that the union of disjoint $K_{r+1}$ is extremal since for every vertex $v$ the construction attains the maximum number of copies of $K_t$ containing $v$.
By looking to the neighbourhood of any vertex, the following cases are immediate as well.

\begin{prop}
	The quantity $\ex(n,K_{1,s},K_{1,r+1})$ is maximized by any $r$-regular graph on $n$ vertices.
	
	For every tree $T$ with maximum degree at most $r$ and diameter $d$, the quantity $\ex(n,T,K_{1,r+1})$ is maximized by any $r$-regular graph of girth at least $d+1$.
\end{prop}

Note that if $nr$ is odd, an extremal graph will have exactly one vertex with degree $r-1$.

When $n$ is not a multiple of $r+1$, Chase's theorem~\cite{C19} (formerly the conjecture of Gan, Loh and Sudakov~\cite{GLS}) implies that the extremal graph is the union of the maximum number of copies of $K_{r+1}$, being the unique graph maximizing $\frac{\ex(n,K_{t},K_{1,r+1})}{n}$ and a residue graph (which is a complete graph as well). The maximum of this normalized quantity can be found easily for complete bipartite graphs as well by looking locally to the neighbourhood of any vertex.

\begin{prop}
For every $r \ge a,b\ge 2$, $\frac{\ex(n,K_{a,b},K_{1,r+1})}{n}$ is maximized by the graph $K_{r,r}$. Furthermore this is the unique connected extremal graph for the quantity.
\end{prop}

In particular we know the extremal graphs for the quantity $\frac{\ex(n,H,K_{1,r+1})}{n}$ when $H \in \{C_3, C_4\}.$
So one can wonder about cycles in general.

\begin{q}\label{q:MaxnormalizedNrCycle}
	For every even cycle $C_m$, for sufficiently large $r$, is $\frac{\ex(n,C_m,K_{1,r+1})}{n}$ maximized by the graph $K_{r,r}$?\\
	For every odd cycle $C_m$, for sufficiently large $r$, is $\frac{\ex(n,C_m,K_{1,r+1})}{n}$ maximized by the graph $K_{r+1}$?
\end{q}

If the latter question is positive for the cycle $C_5$, the following proposition would imply that the analogue of Chase's theorem would not hold, as e.g.~$K_{r+1}+K_1$ is not necessarily maximizing the number of $C_5$s for $n=r+2$ as the following analysis shows.

\begin{prop}
 
Let $r \ge 6$. Then
\[
\ex(r+2,C_5,K_{1,r+1}) = \left.
\begin{cases}
12 \binom{r+1}5 & \text{for $r$ being odd. } \\
\frac{r(r^2-4)(r^2-5r+9)}{10} & \text{for $r$ being even. }
\end{cases}\right.
\]
The extremal graphs are respectively $K_{r+1}$ and $K_{r+2} \backslash M$ for a matching $M$.
\end{prop}

\begin{proof}
	We start with some observations to get some structure of the extremal graphs.
	Note that a graph $G$ of order $n=r+2$ has maximum degree at most $r$ if and only if the complement $\overline G$ has minimum degree $1$.
	If $H$ is a subgraph of $G$, then the number of $C_5$s in $G$ is at least the number of $C_5$s in $H.$
	So if $\overline G$ has an edge for which both of its endvertices have degree at least $2$, we can delete that edge without decreasing the number of $C_5$s in $G.$
	Repeating this, we end with $\overline G$ being the disjoint union of stars.
	Let $\overline G= \sum_{i=1}^k S_{a_i +1}.$ Here $a_i \ge 1$ for every $i$. 
	Note that $A=\sum_i a_i =n-k$ and $k \le \frac n2.$
	Using the principle of inclusion--exclusion, we find that the number of $C_5$s in $G$ equals 
	\begin{equation}\label{eq:nrC5PIE}
	12\binom{n}5 - 6A \binom{n-2}3 + 2\sum_i \binom{a_i}2 \binom{n-3}2 +2 \sum_{i \not=j} a_i a_j (n-4) - 2\sum_{i \not=j} \binom{a_i}2 a_j.
	\end{equation}

	\begin{claim}
		Let $n \ge 9$. For fixed $k$, Equation~\eqref{eq:nrC5PIE} attains its maximum over all $a_i \ge 1$ if and only if all but at most one $a_i$ are equal to $1$. 
	\end{claim}
\begin{proof}
	Note that this is obviously true for $k=1$. Also we note that $A=n-k$ is fixed.
	Now assume $k \ge 2$ and $a_i, a_j >1.$
	The part of Equation~\eqref{eq:nrC5PIE} which depends on $a_i$ and $a_j$ for fixed sum $a_i+a_j$, equals
	\begin{align*}
	&\left((n-3)(n-4) -2A\right)\left( \binom{a_i}2 + \binom{a_j}2 \right) +4a_ia_j(n-4) +2a_i \binom{a_i}2 + 2a_j \binom{a_j}2\\
	=&\left((n-3)(n-4) -2A-4(n-5)\right)\left( \binom{a_i}2 + \binom{a_j}2 \right) +\\ &4(n-5)\binom{a_i+a_j}2 + 4a_ia_j +2a_i \binom{a_i}2 + 2a_j \binom{a_j}2.\\
	\end{align*}
	We have 
	$$(n-3)(n-4) -2A-4(n-5)\ge (n-3)(n-4) -2(n-2)-4(n-5)=n^2-13n+36$$ which is non-negative for $n \ge 9.$
	Also $\binom{a_i+a_j -1}2+\binom{1}{2} > \binom{a_i}2+\binom{a_j}2$ when $a_i, a_j >1$ since $x(x-1)$ is a strictly convex function.
	Furthermore, let $f(x,y)=4xy +2x \binom{x}2 + 2y \binom{y}2.$ Then $f(a_i+a_j -1,1)-f(a_i,a_j)=3(a_i-1)(a_j-1)(a_i+a_j-1)>0.$
	So substituting $(a_i,a_j)$ by $(a_i+a_j-1,1)$ implies an increase of Equation~\eqref{eq:nrC5PIE} from which the result follows.
\end{proof}

	Now we can focus on $a_1=a_2=\ldots=a_{k-1}=1$ and $a_k=n-2k+1.$
	In this case Equation~\eqref{eq:nrC5PIE} reduces to 
	\begin{align*}
	g(n,k)=&12\binom{n}5 - 6(n-k) \binom{n-2}3 + 2 \binom{n-2k+1}2 \binom{n-3}2 \\&+4 \left( \binom{k-1}2 + (k-1)(n-2k+1)\right) (n-4) - 2(k-1) \binom{n-2k+1}{2}.
	\end{align*}
	Note that $\frac{d^2 g(n,k)}{dk^2}=4n^2-24k-32n+108 \ge 4n^2-44n+108=4(n-2)(n-9)+36$ is positive for $n \ge 9.$
	This implies that $g(n,k)$ is strictly convex and hence takes its maximum at $k=1$ or $k= \lfloor \frac n2 \rfloor.$
	Since $g(n, \frac n2) > g(n,1)> g(n, \frac {n-1}2)$ for $ n\ge 9$, we conclude.
	For every edge in both $K_{n-1}$ and $K_{n} \backslash M$ there is a $C_5$ containing that edge, from which we conclude that the extremal graphs are unique in these cases.
	For $n \le 8$, the extremal graphs can easily be computed with computer software such as Sage.	
\end{proof}

\section{Conclusion}

Our work was motivated by a Gan--Loh--Sudakov-type problem where we are given both the number of edges and vertices, in addition to the maximum degree, after~\cite{KR}. By focusing on the regular case and looking to the complement of the extremal graphs, this led us to the notion of regular Tur\'an numbers. This has resulted in a number of interesting regular versions of classical Tur\'an-type results.

The Gan--Loh--Sudakov conjecture was solved by Chase~\cite{C19} in $2019$, very shortly after the appearance of our manuscript on arXiv.
Also the size variation proposed in~\cite{KR} was solved by Chakraborti and Chen~\cite{CC20}.
The variations which we considered here are still open. We note that the main ingredients used in~\cite{GLS,C19,CC20} are not enough to tackle the problem with both order and size.
In particular, our observations in Section~\ref{sec:givenNM} show that the extremal graphs are not that easily described so as to start computations in an inductive way.

Some related questions have arisen, which we suspect should provoke further investigations, particularly with respect to the regular Tur\'an numbers.
In particular, it would be interesting to resolve Question~\ref{q:dich_girth_ge5}, as this would more precisely characterise the regular Tur\'an numbers for graphs of chromatic number $3$. It would also be natural to investigate bipartite graphs.

We also highlight Conjecture~\ref{conj:exactSat}, which would imply both the exact saturation result of the regular Mantel's theorem and the exact form in the regular case of the Gan--Loh--Sudakov-type question given both the order and size.
A last natural problem is Question~\ref{q:MaxnormalizedNrCycle}, being morally the right Gan--Loh--Sudakov question for cycles instead of cliques.

\subsection*{Notes added}
During the preparation of this manuscript, we learned of the concurrent and independent works by Gerbner, Patk\'os, Tuza and Vizer~\cite{GPTV19} and by Caro and Tuza~\cite{CT19}.
With a different application in mind, the regular Tur\'an number was introduced in~\cite{GPTV19} in an alternative formulation as the maximum number of edges in a regular $H$-free graph: $\rex(n,H)=\frac 2n \ex_r(n,H)$.
Caro and Tuza~\cite{CT19} have also determined the regular Tur\'an numbers of complete graphs.

We point out that Theorem~\ref{thr:ex_r_OddCycles} proves Conjecture~$1$ of~\cite{CT19} for large $n$.
For small $n$, the conjecture is false. For example $C_m$ does not contain a $C_g$ when $g<m<2g+4$ nor does the disjoint union of $b$ cliques $K_{m}$ do for $m<g$, leading to another counterexample when $bm<\frac{(m-1)}2(g+2)$.
We also note that Theorem~\ref{thr:RegularErdosStone3} provides progress towards Problem~$1$ in~\cite{CT19}. This problem has been reduced to a more concrete form in Question~\ref{q:dich_girth_ge5}.

A number of relevant works have appeared following the public posting of our manuscript. Besides the works of Chase~\cite{C19} and Chakraborti and Chen~\cite{CC20} already mentioned, we point out that Tait and Timmons~\cite{TT20} have already made a first effort at investigating the regular Tur\'an numbers for bipartite graphs.

\bibliographystyle{abbrv}
\bibliography{RegularTuranNumbers}

\end{document}